\documentclass[reqno]{amsart}
\usepackage{latexsym}
\usepackage{amsfonts}
\usepackage{amssymb}
\usepackage[mathscr]{euscript}

\newtheorem*{question}{Question}
\newtheorem{theorem}{Theorem}[section]
\newtheorem{corollary}[theorem]{Corollary}
\newtheorem{lemma}[theorem]{Lemma}
\newtheorem{proposition}[theorem]{Proposition}
\theoremstyle{definition}
\newtheorem{definition}{Definition}[section]
\newtheorem{example}{Example}[section]
\newtheorem{remark}{Remark}[section]
\numberwithin{equation}{section}

\newcommand{\RR}{\mathbb{R}}
\newcommand{\CC}{\mathbb{C}}

\newcommand{\ZZ}{\mathbb{Z}}
\newcommand{\ZZZ}{\mathbb{Z}/2\mathbb{Z}}

\newcommand{\MM}{\mathbb{M}}
\newcommand{\HH}{\mathbb{H}}
\newcommand{\FF}{\mathbb{F}}

\newcommand{\UU}{\mathscr{U}}
\newcommand{\VV}{\mathscr{V}}
\newcommand{\WW}{\mathscr{W}}

\newcommand{\XX}{\mathscr{X}}
\newcommand{\SSS}{\mathcal{S}}
\newcommand{\II}{\mathcal{I}}
\newcommand{\JJ}{\mathcal{J}}
\newcommand{\HHH}{\mathcal{H}}

\newcommand{\LL}{\mathscr{L}}
\newcommand{\OO}{\mathcal{O}}

\newcommand{\Lg}{\mathfrak{g}}
\newcommand{\Lt}{\mathfrak{t}}
\newcommand{\Lm}{\mathfrak{m}}
\newcommand{\Ln}{\mathfrak{n}}

\newcommand{\ind}{\mathrm{ind}}

\newcommand{\sgn}{\mathrm{sgn}}

\newcommand{\diag}{\mathrm{diag}}
\newcommand{\spa}{\mathrm{span}}
\newcommand{\im}{\mathrm{Im}}

\newcommand{\ii}{\sqrt{-1}}

\def\DJ{{\hbox{D\kern-.8em\raise.15ex\hbox{--}\kern.35em}}}
\def\DJo{\DJ okovi\'c}
\def\DZD{D.\v{Z}. \DJo}

\begin{document}

\title[Universal subspaces]
{Universal subspaces for compact Lie groups}

\author[J. An]{Jinpeng An}
\address{Department of Pure Mathematics, University of
Waterloo, Waterloo, Ontario, N2L 3G1, Canada}
\address{Current address: LMAM, School of Mathematical Sciences, Peking University, Beijing, 100871, China}
\email{anjinpeng@gmail.com}

\author[D.\v{Z}. \DJ okovi\'{c}]{Dragomir \v{Z}. \DJ okovi\'{c}}
\address{Department of Pure Mathematics, University of
Waterloo, Waterloo, Ontario, N2L 3G1, Canada}
\email{djokovic@uwaterloo.ca}

\subjclass[2000]{22C05, 22E45, 57T15, 15A21.}

\keywords{Compact Lie groups and their representations,
characteristic classes, cohomology of flag manifolds, invariant
theory of Weyl groups, matrix zero patterns.}

\thanks{The second author was supported in part by an NSERC Discovery
Grant.}

\begin{abstract}
For a representation of a connected compact Lie group $G$ in a
finite dimensional real vector space $\UU$ and a subspace $\VV$ of
$\UU$, invariant under a maximal torus of $G$, we obtain a
sufficient condition for $\VV$ to meet all $G$-orbits in $\UU$,
which is also necessary in certain cases. The proof makes use of the
cohomology of flag manifolds and the invariant theory of Weyl
groups. Then we apply our condition to the conjugation
representations of $U(n)$, $Sp(n)$, and $SO(n)$ in the space of
$n\times n$ matrices over $\CC$, $\HH$, and $\RR$, respectively. In
particular, we obtain an interesting generalization of Schur's
triangularization theorem.
\end{abstract}

\maketitle

\section{Introduction}

Let $\rho$ be a representation of a connected compact Lie group $G$
in a finite dimensional real vector space $\UU$. We say that a
linear subspace $\VV\subset \UU$ is universal if every $G$-orbit
meets $\VV$. Examples of universal subspaces occur in many places.
For instance, when $\rho$ is irreducible and nontrivial, it is easy
to show (see \cite{GHV}) that every hyperplane of $\UU$ is
universal. For the adjoint representation of $G$ in its Lie algebra
$\Lg$, it is well known that every Cartan subalgebra is universal.
If we consider the complexified adjoint representation of $G$ in the
complexification $\Lg_\CC$ of $\Lg$, then every Borel subalgebra of
$\Lg_\CC$ is universal, and if moreover $G$ is semisimple, then the
sum of all root spaces in $\Lg_\CC$ is also universal (\cite{DT}).
Questions of similar nature have been recently studied in the
context of representation theory of algebraic groups, see e.g.
\cite{GHV,KW}.

In this paper we consider the case where the stabilizer of $\VV$ in
$G$ contains a maximal torus $T$ of $G$. Our main result (Theorem
\ref{T:main}) gives a sufficient condition for the universality of
$\VV$, which is easy to check in concrete examples. With an
additional restriction on $\VV$ this condition becomes also
necessary (see Theorem \ref{T:inverse}).

Our method goes as follows. Consider the trivial vector bundle
$E_\UU=G/T\times\UU$ over the flag manifold $G/T$. Since $\VV$ is
$T$-invariant, there is a well-defined subbundle $E_\VV$ of $E_\UU$
whose fiber at $gT$ is $\rho(g)(\VV)$. Every vector $u\in\UU$,
viewed as a constant section of $E_\UU$, induces a section $s_u$ of
the quotient bundle $E_\UU/E_\VV$. It is easy to see that the
$G$-orbit of $u$ meets $\VV$ if and only if $s_u$ has a zero. So
$\VV$ is universal if and only if $s_u$ has a zero for every
$u\in\UU$. A sufficient condition for this is that $E_\UU/E_\VV$ is
orientable and has a nonzero Euler class. If we choose a
$T$-invariant subspace $\WW$ of $\UU$ complementary to $\VV$ and
construct similarly the subbundle $E_\WW$ of $E_\UU$, then
$E_\WW\cong E_\UU/E_\VV$. It is easy to see that $E_\WW$ is always
orientable and is equivalent to $G\times_T\WW$, the associated
bundle induced by the $T$-action on $\WW$. By using a theorem of
Borel on the cohomology of $G/T$ and the invariant theory of Weyl
groups, we can compute explicitly the Euler class $e(E_\WW)$ of
$E_\WW$.

Let $\XX(T)$ be the character group of $T$, and let $\SSS$ be the
graded algebra of real polynomial functions on the Lie algebra of
$T$. The map from $\XX(T)$ to $H^2(G/T,\RR)$ sending $\chi\in\XX(T)$
to the first Chern class of the complex line bundle
$G\times_\chi\CC$ induces a homomorphism $\varphi:\SSS\rightarrow
H^*(G/T,\RR)$ of graded algebras, provided we double the degree in
$\SSS$. Borel's theorem asserts that $\varphi$ is surjective and its
kernel is the ideal $\II$ of $\SSS$ generated by the basic
generators of $W$-invariant polynomials, where $W=N_G(T)/T$ is the
Weyl group. If we assume that the above subspace $\WW$ of $\UU$ does
not contain nonzero $T$-invariant vectors, then $\WW$ can be
decomposed as a direct sum $\bigoplus_{i=1}^d\WW_i$ of
$2$-dimensional $T$-irreducible subspaces. The $T$-action on every
$\WW_i$ is equivalent to the real $T$-module associated to some
$\chi_i\in\XX(T)$, and then $G\times_T\WW_i$ is equivalent to the
real vector bundle associated to $G\times_{\chi_i}\CC$. So
$e(E_\WW)$ is equal to the product of the first Chern classes of
$G\times_{\chi_i}\CC$, i.e., $\prod_{i=1}^d\varphi(\chi_i)$. We
refer to the homogeneous polynomial $f_\VV=\prod_{i=1}^d\chi_i$ as
the characteristic polynomial of $\VV$. So $e(E_\WW)\neq0$ if and
only if $f_\VV\notin\II$, and we obtain the desired sufficient
condition for the universality of $\VV$. In general, the question of
whether $f_\VV$ is in $\II$ can be resolved using Gr\"{o}bner bases.

An easy dimension argument shows that if $\VV$ is universal, then
$\dim\VV\geq\dim\UU-\dim G/T$. The case of equality is of particular
interest (indeed, for many representations, a universal subspace
$\VV$ whose universality can be detected using our method can be
shrunk to a smaller universal subspace for which the equality
holds). In this case, the rank of $E_\WW$ is equal to $\dim G/T$,
and $f_\VV$ has degree $m=\frac{1}{2}\dim G/T$. By invariant theory
of Weyl groups, a homogeneous polynomial of degree $m$ is in $\II$
if and only if it is orthogonal to the fundamental harmonic
polynomial $f_0$ (see \eqref{E:f0}) with respect to a $W$-invariant
inner product on $\SSS$. So if we define the characteristic number
$C_\VV$ of $\VV$ as the inner product of $f_\VV$ and $f_0$ (up to a
constant factor, which will be chosen in a way suitable for the
computation of intersection numbers), then $f_\VV\notin\II$ if and
only if $C_\VV\neq0$, which gives a simpler equivalent sufficient
condition for the universality of $\VV$. Moreover, if we assume that
the intersection indexes of the above section $s_u$ and the zero
section $s_0$ of $E_\WW$ are all equal whenever $s_u$ is transversal
to $s_0$, then $\VV$ is universal if and only if $C_\VV\neq0$.

In concrete examples that we shall consider, the characteristic
polynomials and numbers are easy to compute. We shall do this for
the conjugation representations of $U(n)$, $Sp(n)$, and $SO(n)$ in
the spaces $\MM(n,\FF)$ of $n\times n$ matrices over $\FF=\CC$,
$\HH$, and $\RR$, respectively, for subspaces of $\MM(n,\FF)$
determined by zero patterns. A zero pattern is a subset $I$ of
$\{(i,j)|i,j\in\{1,\ldots,n\},i\neq j\}$ of cardinality $n(n-1)/2$.
For $\FF=\CC$ or $\HH$ and a zero pattern $I$, we denote by
$\MM_I(n,\FF)$ the subspace of $\MM(n,\FF)$ consisting of matrices
whose $(i,j)$-th entry is $0$ whenever $(i,j)\in I$. We shall apply
the above result to $\MM_I(n,\FF)$ and obtain a sufficient condition
for the universality of $\MM_I(n,\FF)$ expressed as a simple
property of $I$. In particular, we show that if $I$ contains exactly
one of $(i,j)$ and $(j,i)$ for $i\neq j$ (we call such $I$ a simple
zero pattern), then $\MM_I(n,\FF)$ is universal. Similar results are
also obtained for $\FF=\RR$. Note that if $I_0=\{(i,j)|1\leq i<j\leq
n\}$, $\MM_{I_0}(n,\CC)$ is the subspace of all complex lower
triangular matrices, its universality is known as Schur's
triangularization theorem \cite{IS}. So our result gives a genuine
generalization of Schur's theorem. If $\FF=\CC$ and $I$ is a
bitriangular zero pattern (defined in Section \ref{Schur}), the
above sufficient condition for the universality of $\MM_I(n,\CC)$ is
also necessary, and the number of flags in $U(n)/U(1)^n$ sending a
generic matrix into $\MM_I(n,\CC)$ can be computed directly from
$I$.


The problem of universality of subspaces of $\MM(n,\CC)$ determined
by zero patterns has been raised by researchers in linear algebra.
For instance, the universality of $\MM_I(4,\CC)$ when
$I=\{(1,3),(1,4),(2,4),(3,1),(4,1),(4,2)\}$ has been studied by
several authors. It was shown by Pati \cite{VP} that it is
universal. A simpler proof was given later in \cite{DD} where it is
also shown that the number of flags sending a generic matrix
$A\in\MM(4,\CC)$ into $\MM_I(n,\CC)$ is equal to $12$. This zero
pattern is bitriangular and so both results follow immediately from
Theorem \ref{T:A-inverse}. Another simple example is the problem of
universality of $\MM_I(3,\CC)$ where $I=\{(1,3),(2,1),(3,2)\}$. This
was raised a few years ago \cite{DJ} but remained unsettled until
now. In our terminology $I$ is a simple zero pattern, so
$\MM_I(3,\CC)$ is universal (see Theorem \ref{T:A2} and the
paragraph preceding it). For further information along this
direction see e.g. \cite{DZ,HS}.

The arrangement of this paper is as follows. In Section 2 we shall
recall some preliminaries on Euler classes, intersection indexes,
cohomology of flag manifolds, and invariant theory of Weyl groups.
The definitions of the characteristic polynomial and the
characteristic number of a subspace and their properties will be
given in Section 3. The main results will be proved in Section 4. In
Section 5 we shall study the universality of subspaces determined by
zero patterns, generalizing Schur's triangularization theorem. Then
in Section 6 we shall briefly discuss the possibility of
generalizing our method and propose a question.

The first author would like to thank Jiu-Kang Yu for valuable
conversations.

\section{Preliminaries}

In this section we recall some basic facts about Euler classes of
oriented vector bundles, cohomology of the flag manifold $G/T$, and
invariant theory of Weyl groups. Throughout this paper, we denote by
$G$ a connected compact Lie group and by $T$ a maximal torus of $G$.

\subsection{Euler classes and intersection indexes}

Let $M$ be a connected compact smooth $n$-dimensional manifold, and
let $E$ be a smooth real oriented vector bundle over $M$ of rank
$r$. It is well known that if the Euler class $e(E)\in H^r(M,\ZZ)$
of $E$ is nonzero, then every smooth section of $E$ has a zero. If
$E=E_1\oplus E_2$ as oriented vector bundles, then
$e(E)=e(E_1)e(E_2)$. If $E$ admits a complex structure compatible
with the orientation, then the top Chern class $c_{\mathrm{top}}(E)$
of $E$ is equal to $e(E)$ (see \cite{BT,MS}).

Now we assume that $M$ is oriented and $r=n$. Let $s_0$ be the zero
section of $E$. Then for any smooth section $s$ of $E$, we have
$e(E)([M])=I(s,s_0)$, where $[M]\in H_n(M,\ZZ)$ is the fundamental
class of $M$ and $I(s,s_0)$ is the intersection number of $s$ and
$s_0$. Let $Z(s)=\{x\in M|s(x)=0\}$ be the zero locus of $s$. We
write $s\pitchfork_xs_0$ if $s$ is transversal to $s_0$ at the point
$x\in Z(s)$, and $s\pitchfork s_0$ if $s\pitchfork_xs_0$ for all
$x\in Z(s)$. If $s\pitchfork s_0$, then $Z(s)$ is finite and
$I(s,s_0)=\sum_{x\in Z(s)}\ind_x(s)$, where $\ind_x(s)=\pm1$ is the
intersection index of $s$ and $s_0$ at $x$. Assume that $E$ is a
subbundle of a trivial bundle $M\times\UU$, where $\UU$ is a real
vector space, and let us view $s$ as a map $M\rightarrow\UU$. For
$x\in Z(s)$, we can compute $\ind_x(s)$ as follows. Since
$s\pitchfork_xs_0$, the tangent map $(ds)_x:T_xM\rightarrow
T_0\UU\cong\UU$ is injective and has image $E_x$, the fiber of $E$
at $x$. If we view $(ds)_x$ as an invertible map $T_xM\rightarrow
E_x$, then $\ind_x(s)=\sgn((ds)_x)$. (Here the sign $\sgn(A)$ of an
invertible linear map $A:V_1\rightarrow V_2$ between real oriented
vector spaces refers to the sign of the determinant of the matrix of
$A$ with respect to any ordered bases of $V_1$ and $V_2$ compatible
with the orientations.)

We remark that if $E$ is the tangent bundle $TM$ of $M$, then
$e(TM)([M])$ is the Euler characteristic of $M$.

In this paper, $M$ will almost exclusively be the flag manifold
$G/T$. We shall need some facts about the cohomology of $G/T$, which
we now recall.

\subsection{Cohomology of $G/T$}
Consider the maximal torus $T$ of the connected compact Lie group
$G$. Let $\XX(T)$ be the character group of $T$. Every character
$\chi\in \XX(T)$, viewed as a representation of $T$ in $\CC$,
induces a complex line bundle $L_\chi=G\times_\chi\CC$ over $G/T$,
where the total space of $L_\chi$ is the space of $T$-orbits in
$G\times\CC$ with the $T$-action $t\cdot(g,z)=(gt^{-1},\chi(t)z)$.
Let $c_1(L_\chi)\in H^2(G/T,\ZZ)$ be the first Chern class of
$L_\chi$. Then the map sending $\chi$ to $c_1(L_\chi)$ induces a
homomorphism of abelian groups
\begin{equation}\label{E:homo1}
\XX(T)\rightarrow H^2(G/T,\ZZ).
\end{equation}
Note that $\XX(T)$ can be identified with a lattice in
$\mathfrak{t}^*$, the dual of the Lie algebra $\mathfrak{t}$ of $T$,
by using differentials. It will be more convenient to work with
cohomology with real coefficients. (As $H^*(G/T,\ZZ)$ is torsion
free (\cite{MM}, Chapter 6, Theorem 4.21), no information is lost
when $\ZZ$ is replaced by $\RR$.) So we consider the linear map
\begin{equation}\label{E:homo2}
\mathfrak{t}^*\cong\XX(T)\otimes_\ZZ\RR\rightarrow H^2(G/T,\RR)\cong
H^2(G/T,\ZZ)\otimes_\ZZ\RR
\end{equation}
induced by the homomorphism \eqref{E:homo1}.

Let $\SSS$ be the symmetric algebra of $\Lt^*$, identified with the
graded algebra of real polynomial functions on $\Lt$. Let
$H^*(G/T,\RR)$ be the graded algebra of real cohomology of $G/T$.
The homomorphism \eqref{E:homo2} induces an algebra homomorphism
\begin{equation}\label{E:homo}
\varphi:\SSS\rightarrow H^*(G/T,\RR)
\end{equation}
which doubles the degree. Let $W=N_G(T)/T$ be the Weyl group. Then
$W$ acts naturally on $\Lt, \Lt^*$ and $\SSS$. Let $\JJ$ be the
graded algebra of $W$-invariants in $\SSS$, let $\SSS_d$ be the
homogeneous component of degree $d$ in $\SSS$, and let
$\JJ_d=\JJ\cap\SSS_d$ and $\JJ_+=\bigoplus_{d\geq1}\JJ_d$. Let
$\II=\SSS\JJ_+$ be the ideal of $\SSS$ generated by $\JJ_+$ and set
$\II_d=\II\cap\SSS_d$.

\begin{theorem}[Borel \cite{Bo}, Theorem 26.1]\label{T:Borel}
The above homomorphism $\varphi$ is surjective and
$\ker(\varphi)=\II$. Consequently, $H^*(G/T,\RR)\cong \SSS/\II$ as
graded algebras, provided we double the degrees in $\SSS/\II$.
\end{theorem}

Although Borel did not mention Chern classes explicitly in his
theorem in \cite{Bo}, his presentation is equivalent to ours. Chern
classes are explicitly used in the approaches to Borel's Theorem in
\cite{BGG, Fu, KiWe, Re, Yu}. Our construction of the homomorphism
$\varphi$ follows from \cite{KiWe}, Chapter VI, Theorem 2.1 and
\cite{Yu} (see also \cite{Fu}, Section 10.2, Proposition 3). This
homomorphism can be constructed also by using classifying spaces.
Let $BK$ denote the classifying space of a group $K$. Then $\varphi$
is equal to the composition of the Chern-Weil isomorphism
$\SSS\rightarrow H^*(BT,\RR)$ and the homomorphism
$H^*(BT,\RR)\rightarrow H^*(G/T,\RR)$ induced by the fiber bundle
$G/T\rightarrow BT\rightarrow BG$ (see \cite{Du,MM}).

At last, we recall the well known fact that the Euler characteristic
$\chi(G/T)$ of $G/T$ is equal to $|W|$ (see e.g. \cite{MM}, Chapter
5, Theorem 3.14).

\subsection{Invariant theory of Weyl groups}

We have seen above that $H^*(G/T,\RR)$ can be described by
$W$-invariants in the symmetric algebra $\SSS$. To understand
$H^*(G/T,\RR)$ more closely, we now recall some basic facts about
invariant theory of the Weyl group $W$. For details see
\cite{He,Ka,Sm}.

For $H\in\Lt$, the partial derivative $\partial_Hf$ of $f\in \SSS$
is defined by $(\partial_Hf)(H')=\frac{d}{dt}\big|_{t=0}f(H'+tH)$.
Let $(\cdot,\cdot)$ be a $W$-invariant inner product on $\Lt$. Then
there is a $W$-equivariant isomorphism $\iota:\Lt^*\rightarrow\Lt$
defined by the identity $(\iota(\alpha),H)=\alpha(H)$, under which
the map $H\mapsto\partial_H$ extends to an isomorphism
$f\mapsto\partial_f$ of $\SSS$ onto the algebra of differential
operators on $\Lt$ with constant coefficients. The bilinear form $B$
on $\SSS$ defined by $B(f,g)=(\partial_fg)(0)$ is a $W$-invariant
inner product on $\SSS$. If
$\alpha_1,\ldots,\alpha_k,\beta_1,\ldots,\beta_k\in\Lt^*$, then
\begin{equation}\label{E:B}
B(\alpha_1\cdots\alpha_k,\beta_1\cdots\beta_k)=\sum_{\sigma\in
S_k}\prod_{i=1}^k (\iota(\alpha_i),\iota(\beta_{\sigma(i)})),
\end{equation}
where $S_k$ is the symmetric group on $k$ letters. It is easily seen
that $\partial_f$ is the adjoint operator of the multiplication
operator by $f$, that is, $B(g,fh)=B(\partial_fg,h)$ for any
$f,g,h\in\SSS$.

A polynomial $f\in \SSS$ is \emph{$W$-harmonic} if $\partial_gf=0$
for all $g\in \JJ_+$. Let $\HHH$ denote the space of $W$-harmonic
polynomials, and let $\HHH_d=\HHH\cap\SSS_d$. Let
$\Phi\subset\XX(T)\subset\Lt^*$ be the root system of $G$, $\Phi^+$
a system of positive roots, and $m=|\Phi^+|=\frac{1}{2}\dim G/T$.
Some basic facts that we are going to use are collected in the
following theorem.

\begin{theorem}\label{T:invariant}
(1) There are $r=\dim\Lt$ basic homogeneous $W$-invariant
polynomials $F_1,\ldots,F_r$ which are algebraically independent
such that the subalgebra of $\SSS$ generated by $F_1,\ldots,F_r$ is
$\JJ$, and the ideal of $\SSS$ generated by $F_1,\ldots,F_r$ is $\II=\SSS\JJ_+$.\\
(2) The degrees $d_i=\deg F_i$ are determined by $W$,
$\prod_{i=1}^rd_i=|W|$, $\sum_{i=1}^r(d_i-1)=m$.\\
(3) $\SSS=\II\oplus\HHH$ orthogonally with respect to $B$, and the
Poincar\'{e} polynomial of $\SSS/\II$ is given by
$\sum_{d=1}^\infty(\dim \HHH_d)
t^d=\prod_{i=1}^r(1+t+\cdots+t^{d_i-1})$.
In particular, $\HHH_d=0$ for $d>m$ and $\dim \HHH_m=1$.\\
(4) The polynomial
\begin{equation}\label{E:f0}
f_0=\prod_{\alpha\in\Phi^+}\alpha
\end{equation}
is in $\HHH_m$, and $\HHH=\{\partial_f f_0|f\in \SSS\}$.\\
(5) For $f\in \SSS$, $f\in \II$ if and only if $\partial_f f_0=0$.
In particular,
for $f\in\SSS_m$ we have $f\in \II_m$ if and only if $B(f,f_0)=0$.\\
(6) $f\in \SSS$ is $W$-skew if and only if $f\in f_0\JJ$. (Here $f$
is called $W$-skew if $w\cdot f=\sgn(w)f$ for every $w\in W$.)
\end{theorem}

We refer to the polynomial $f_0\in\HHH_m$ given in \eqref{E:f0} as
the \emph{fundamental harmonic polynomial}. Given $f\in\SSS_m$, it
is important to know when $f\in\II_m$. Theorem \ref{T:invariant} (5)
provides a characterization using $f_0$ and the inner product $B$.
Although it is not difficult to compute $B$ using \eqref{E:B}, the
following generalization of Theorem \ref{T:invariant} (5) will be
more convenient.

\begin{proposition}\label{P:inner}
$\II_m=f_0^\bot$ with respect to any $W$-invariant inner product on
$\SSS_m$.
\end{proposition}

\begin{proof}
By Theorem \ref{T:invariant} (6), every $W$-skew polynomial in
$\SSS_m$ is a multiple of $f_0$. So for the natural representation
of $W$ in $\SSS_m$, the isotypic component of the irreducible
representation $w\mapsto\sgn(w)$ is $\RR f_0$. Hence the orthogonal
complement of $\RR f_0$ is independent of the choice of the
$W$-invariant inner product on $\SSS_m$.
\end{proof}

In view of Proposition \ref{P:inner}, the fundamental harmonic
polynomial $f_0$ plays an important role. For later reference, we
give below explicit expressions for $f_0$ and a $W$-invariant inner
product on $\SSS$ for classical groups. Before doing this, we define
a polynomial $\lambda_0$ in $\RR[x_1,\ldots,x_n]$ as
\begin{equation}\label{E:lambda0}
\lambda_0(x)=\prod_{1\leq i<j\leq n}(x_i-x_j)=\sum_{\sigma\in
S_n}\sgn(\sigma)\prod_{i=1}^{n}x^{n-i}_{\sigma(i)},
\end{equation}
where $x=(x_1,\ldots,x_n)$. This polynomial will appear in all cases
of classical groups.

\begin{example}\label{Ex:A}
Let $G=U(n)$, let $T=U(1)^n=\{t=\diag(t_1,\ldots,t_n)|t_i\in U(1)\}$
be the standard maximal torus and $\mathfrak{t}=\{x=\diag(\ii
x_1,\ldots,\ii x_n)|x_i\in\RR\}$. Then $\SSS$ can be identified with
the polynomial algebra $\RR[x_1,\ldots,x_n]$, and a character
$\chi(t)=t_1^{m_1}\cdots t_n^{m_n}$ with the linear form
$m_1x_1+\cdots+m_nx_n$. The root system is given by
$$\Phi=\{\pm(x_i-x_j)|1\leq i<j\leq n\},$$ and $\Phi^+$ can be chosen as
$$\Phi^+=\{x_i-x_j|1\leq i<j\leq n\}.$$ So in this case
$$f_0(x)=\lambda_0(x).$$ We have
$\deg f_0=|\Phi^+|=n(n-1)/2$ and $W\cong S_n$. There is a unique
$W$-invariant inner product on $\SSS$ for which the monomials
$x_1^{m_1}\cdots x_n^{m_n}$ form an othornormal basis. The basic
generators $F_1,\ldots,F_n$ of $\II$ can be chosen as the elementary
symmetric polynomials in $x_1,\ldots,x_n$. \qed\end{example}

\begin{example}\label{Ex:C}
Let $G=Sp(n)$. $T=U(1)^n$ is a maximal torus and
$\mathfrak{t}=\{x=\diag(\ii x_1,\ldots,\ii x_n)|x_i\in\RR\}$. So
$\SSS\cong\RR[x_1,\ldots,x_n]$. The root system  is given by
$$\Phi=\{\pm x_i\pm x_j|1\leq i<j\leq
n\}\cup\{\pm2x_i|1\leq i\leq n\},$$ and $\Phi^+$ can be chosen as
$$\Phi^+=\{x_i\pm x_j|1\leq i<j\leq n\}\cup\{2x_i|1\leq i\leq n\}.$$
So
$$f_0(x)=2^n\lambda_0(x_1^2,\ldots,x_n^2)\prod_{i=1}^nx_i.$$ We
have $\deg f_0=|\Phi^+|=n^2$ and $W\cong (\ZZ/2\ZZ)^n\rtimes S_n$.
The $W$-invariant inner product on $\SSS$ can be chosen as in
Example \ref{Ex:A}. The basic generators $F_1,\ldots,F_n$ of $\II$
can be chosen as the elementary symmetric polynomials in
$x_1^2,\ldots,x_n^2$. \qed\end{example}

\begin{example}\label{Ex:D}
Let $G=SO(2n)$. $T=SO(2)^n$ is a maximal torus and
$\Lt=\left\{x=\diag\left(\begin{pmatrix}0&x_1\\-x_1&0\\\end{pmatrix},\ldots,
\begin{pmatrix}0&x_n\\-x_n&0\\\end{pmatrix}\right)\right\}$. So
$\SSS\cong\RR[x_1,\ldots,x_n]$. The root system  is given by
$$\Phi=\{\pm x_i\pm x_j|1\leq i<j\leq n\},$$
and $\Phi^+$ can be chosen as $$\Phi^+=\{x_i\pm x_j|1\leq i<j\leq
n\}.$$ So $$f_0(x)=\lambda_0(x_1^2,\ldots,x_n^2).$$ We have $\deg
f_0=|\Phi^+|=n(n-1)$ and $W\cong (\ZZ/2\ZZ)^{n-1}\rtimes S_n$.  The
$W$-invariant inner product on $\SSS$ can be chosen as in Example
\ref{Ex:A}. The basic generators $F_1,\ldots,F_n$ of $\II$ can be
chosen as the first $n-1$ elementary symmetric polynomials in
$x_1^2,\ldots,x_n^2$ and $\prod_{i=1}^nx_i$. \qed\end{example}

\begin{example}\label{Ex:B}
Let $G=SO(2n+1)$. $T=SO(2)^n\times\{1\}$ is a maximal torus and
$\Lt=\left\{x=\diag\left(\begin{pmatrix}0&x_1\\-x_1&0\\\end{pmatrix},\ldots,
\begin{pmatrix}0&x_n\\-x_n&0\\\end{pmatrix},0\right)\right\}$. So
$\SSS\cong\RR[x_1,\ldots,x_n]$. The root system is given by
$$\Phi=\{\pm x_i\pm x_j|1\leq i<j\leq n\}\cup\{\pm x_i|1\leq i\leq n\},$$
and $\Phi^+$ can be chosen as $$\Phi^+=\{x_i\pm x_j|1\leq i<j\leq
n\}\cup\{x_i|1\leq i\leq n\}.$$ So
$$f_0(x)=\lambda_0(x_1^2,\ldots,x_n^2)\prod_{i=1}^nx_i.$$ We have
$\deg f_0=|\Phi^+|=n^2$ and $W\cong (\ZZ/2\ZZ)^n\rtimes S_n$. The
$W$-invariant inner product on $\SSS$ can be chosen as in Example
\ref{Ex:A}. The basic generators $F_1,\ldots,F_n$ of $\II$ can be
chosen to be the same as in Example \ref{Ex:C}. \qed\end{example}

\section{Characteristic polynomials of subspaces} \label{Kar-pol}

Let $\rho$ be a representation of $G$ in a finite dimensional real
vector space $\UU$, and let $\UU^T$ be the subspace of $T$-invariant
vectors in $\UU$. Let $\VV\subset\UU$ be a $T$-invariant subspace
containing $\UU^T$. We shall now define the characteristic
polynomial $f_\VV$ of $\VV$ (up to a sign which depends on the
orientation of $\UU/\VV$).

Let $\WW$ be a $T$-invariant subspace of $\UU$ complementary to
$\VV$. Every real irreducible representation of $T$ is either
trivial or $2$-dimensional. In the latter case, it is equivalent to
the the real $T$-module associated to some nontrivial character
$\chi$ of $T$. Since $\UU^T\subset\VV$, there is no trivial
subrepresentation of $T$ in $\WW$. So $\WW$ can be decomposed as a
direct sum
\begin{equation}\label{E:decompose}
\WW=\bigoplus_{i=1}^d\WW_i, \quad d=\frac{1}{2}\dim\WW,
\end{equation}
of $2$-dimensional $T$-irreducible subspaces. Note that the real
$T$-modules associated to $\chi$ and $-\chi$ are equivalent. To
associate to $\WW_i$ a unique $\chi_i\in\XX(T)\setminus\{0\}$, we
choose and fix an orientation on $\WW_i$. Then there is a unique
$\chi_i\in\XX(T)\setminus\{0\}$ for which there exists a
$T$-equivariant orientation preserving isomorphism
$\sigma_i:(\RR^2,\chi_i)\rightarrow\WW_i$, where $(\RR^2,\chi_i)$ is
the real $T$-module associated to $\chi_i$ with the standard
orientation. The orientations of the $\WW_i$'s induce an orientation
on $\WW\cong\UU/\VV$. It is easy to see that the product of the
$\chi_i$'s depends only on the induced orientation on $\WW$, and
changes sign if the orientation is reversed.

\begin{definition}
Given the orientation on $\WW\cong\UU/\VV$ as above, we define the
\emph{characteristic polynomial} $f_\VV\in\SSS_d$ of $\VV$ by
$$f_\VV=\prod_{i=1}^d\chi_i.$$
\end{definition}

The equivalence $\sigma_i$ above induces a complex structure $J_i$
on $\WW_i$, which is independent of the choice of $\sigma_i$ and is
determined by the orientation and the $T$-action on $\WW_i$. The
direct sum
\begin{equation}\label{E:JW}
J_\WW=\bigoplus_{i=1}^dJ_i
\end{equation}
is a complex structure on $\WW$ compatible with the orientation. For
any nonzero vectors $w_i\in\WW_i$, the ordered basis
\begin{equation}\label{E:basis-W}
\{w_1,J_\WW w_1,\ldots,w_d,J_\WW w_d\}
\end{equation}
of $\WW$ is compatible with the orientation.

\begin{remark}\label{R:1}
If $\UU$ admits a complex structure such that the $G$-action is
complex linear, and if $\VV\subset\UU$ is a complex subspace, then
we can choose $\WW$ to be a complex subspace. Thus $\WW$ has the
complex decomposition
$$\WW=\bigoplus_{\chi\in\XX(T)}\WW_\chi,$$
where $\WW_\chi$ is the $\chi$-isotypic component of $\WW$. As
$\UU^T\subset\VV$, we have $\WW_0=0$. Each $\WW_\chi$ can be further
decomposed as a direct sum of complex lines, which serve as the
$\WW_i$'s in \eqref{E:decompose}. If the orientation on each of
these complex lines is chosen to be the one induced by the original
complex structure, then
$$f_\VV=\prod_{\chi\in\XX(T)}\chi^{\dim_\CC\WW_\chi},$$
and $J_\WW$ coincides with the original complex structure on $\WW$.
\end{remark}

We now consider the trivial real vector bundle $E_\UU=G/T\times\UU$
over the flag manifold $G/T$. For any $T$-invariant subspace $\LL$
of $\UU$, there is a well-defined subbundle $E_\LL$ of $E_\UU$ whose
fiber at $gT$ is $\rho(g)(\LL)$. In particular, $\VV$ and $\WW$ are
$T$-invariant, and so we have the subbundles $E_\VV$ and $E_\WW$.
Since $\UU=\rho(g)(\VV)\oplus\rho(g)(\WW)$ for every $g\in G$, we
have $E_\UU=E_\VV\oplus E_\WW$. The $T$-action on $\WW$ preserves
the orientation. So the orientation on $\WW$ induces an orientation
on $E_\WW\cong E_\UU/E_\VV$, and we can consider the Euler class
$e(E_\WW)=e(E_\UU/E_\VV)\in H^{2d}(G/T,\RR)$. Since $J_\WW$ is
$T$-invariant, it induces a complex structure on $E_\WW$. So we can
also consider the top Chern class $c_\mathrm{top}(E_\WW)\in
H^{2d}(G/T,\RR)$. Since the orientation and the complex structure
are compatible, $e(E_\WW)=c_\mathrm{top}(E_\WW)$.

\begin{theorem}\label{T:class}
We have $e(E_\UU/E_\VV)=\varphi(f_\VV)$, where $\varphi$ is the
homomorphism \eqref{E:homo}.
\end{theorem}

\begin{proof}
We view $E_\WW$ as a complex vector bundle and prove that
$c_\mathrm{top}(E_\WW)=\varphi(f_\VV)$. Let $L_i$ be the subbundle
of $E_\UU$ whose fiber at $gT$ is $\rho(g)(\WW_i)$. The complex
structure $J_i$ on $\WW_i$ induces a complex structure on $L_i$. By
\eqref{E:decompose}, $E_\WW\cong\bigoplus_{i=1}^dL_i$ as complex
vector bundles. So we have
\begin{equation}\label{E:ctop1}
c_{\mathrm{top}}(E_\WW)=\prod_{i=1}^{d}c_1(L_i),
\end{equation}
where $c_1(L_i)$ is the first Chern class of $L_i$. For
$\chi\in\XX(T)$ let $L_\chi=G\times_\chi\CC$ be the complex line
bundle over $G/T$ induced by $\chi$. We claim that $L_i\cong
L_{\chi_i}$ as complex line bundles. Indeed, given any nonzero
vector $w_i\in\WW_i$, the map $G\times\CC\rightarrow L_i$ defined by
$$(g,z)\mapsto(gT,\rho(g)(zw_i))$$ induces an equivalence
between $L_{\chi_i}$ and $L_i$. So by the definition of $\varphi$,
we have
\begin{equation}\label{E:ctop2}
c_1(L_i)=c_1(L_{\chi_i})=\varphi(\chi_i).
\end{equation}
From \eqref{E:ctop1} and \eqref{E:ctop2} we deduce that
$$c_{\mathrm{top}}(E_\WW)=\prod_{i=1}^d\varphi(\chi_i)
=\varphi(\prod_{i=1}^d\chi_i)=\varphi(f_\VV).$$
\end{proof}

Set $m=\frac{1}{2}\dim G/T$. We shall be mainly interested in the
case $d=m$. In this case, it will be useful to know the explicit
value of $e(E_\UU/E_\VV)([G/T])$, where $[G/T]\in H_{2m}(G/T,\ZZ)$
is the fundamental class of $G/T$ with respect to some orientation
on $G/T$. We fix the orientation on $G/T$ as follows. Let $\Lg$ and
$\Lt$ be the Lie algebras of $G$ and $T$, respectively, and let
$\Phi\subset\XX(T)$ be the root system of $G$ and
$\Phi^+=\{\alpha_1,\ldots,\alpha_m\}$ a system of positive roots.
Let $\Lg_\CC=\Lt_\CC\oplus\bigoplus_{\alpha\in\Phi}\Lg_\alpha$ be
the root space decomposition of the complexification of $\Lg$, and
set $\Lm=\Lg\cap(\bigoplus_{\alpha\in\Phi}\Lg_\alpha)$. Then there
exist $X_\alpha\in\Lg_\alpha$ such that
\begin{equation}\label{E:basis-m}
\{X_{\alpha_1}-X_{-\alpha_1},
\ii(X_{\alpha_1}+X_{-\alpha_1}),\ldots,X_{\alpha_m}-X_{-\alpha_m},
\ii(X_{\alpha_m}+X_{-\alpha_m})\}
\end{equation}
is a basis of $\Lm$. For $\alpha\in\Phi^+$, let
$\Lm_\alpha=\mathrm{span}\{X_\alpha-X_{-\alpha},
\ii(X_\alpha+X_{-\alpha})\}$, and let $J_\alpha$ be the complex
structure on $\Lm_\alpha$ defined by
$$J_\alpha(X_\alpha-X_{-\alpha})=\ii(X_\alpha+X_{-\alpha}), \quad
J_\alpha(\ii(X_\alpha+X_{-\alpha}))=-(X_\alpha-X_{-\alpha}).$$ Then
$\Lm=\bigoplus_{\alpha\in\Phi^+}\Lm_\alpha$, and
$J_\Lm=\bigoplus_{\alpha\in\Phi^+}J_\alpha$ is a complex structure
on $\Lm$, which induces an orientation on $\Lm$ compatible with the
ordered basis \eqref{E:basis-m}. For $g\in G$, define the map
$\tau_g:\Lm\rightarrow G/T$ by
\begin{equation}\label{E:tau}
\tau_g(X)=ge^XT.
\end{equation}
Then the differential $(d\tau_g)_0:\Lm\rightarrow T_{gT}(G/T)$ is
invertible and induces an orientation on $T_{gT}(G/T)$, which
depends only on $gT$. This induces an orientation on $G/T$.

\begin{definition}
If $d=m$, we define the \emph{characteristic number} $C_\VV$ of
$\VV$ by $$C_\VV=\frac{\langle f_\VV,f_0\rangle}{\langle
f_0,f_0\rangle}|W|,$$ where $W$ is the Weyl group,
$\langle\cdot,\cdot\rangle$ is a $W$-invariant inner product on
$\SSS_m$, and $f_0$ is the fundamental harmonic polynomial defined
in \eqref{E:f0}.

\end{definition}

\begin{proposition}\label{P:independent}
$C_\VV$ is independent of the choice of $\langle\cdot,\cdot\rangle$.
\end{proposition}

\begin{proof}
The linear functional $\SSS_m\rightarrow\RR$ sending $f$ to
$\frac{\langle f,f_0\rangle}{\langle f_0,f_0\rangle}|W|$ takes the
value $|W|$ at $f_0$ and vanishes on $f_0^\bot=\II_m$.
\end{proof}

We shall prove that, with respect to the above orientation on $G/T$,
$e(E_\WW)([G/T])=C_\VV$. We first prove a special case. Let $E_\Lm$
be the subbundle of the trivial bundle $G/T \times \Lg$
corresponding to the $T$-invariant subspace $\Lm$.

\begin{lemma}\label{L:1}
For the adjoint representation of $G$ in $\Lg$, we have $f_\Lt=f_0$
and $e(E_\Lm)([G/T])=C_\Lt=|W|$.
\end{lemma}

\begin{proof}
With respect to the complex structure $J_\alpha$, the $T$-action on
$\Lm_\alpha$ is complex and is equivalent to $\alpha$. So
$f_\Lt=\prod_{\alpha\in\Phi^+}\alpha=f_0$ and $C_\Lt=|W|$.

To finish the proof, we note that the orientation preserving map
$$(d\tau_g)_0\circ
\mathrm{Ad}(g)^{-1}:(E_\Lm)_{gT}=\mathrm{Ad}(g)(\Lm)\rightarrow
T_{gT}(G/T)$$ depends only on $gT$, where $\tau_g$ is the map
defined in \eqref{E:tau}. This induces an orientation preserving
equivalence between $E_\Lm$ and the tangent bundle $T(G/T)$. So
$e(E_\Lm)([G/T])=e(T(G/T))([G/T])=\chi(G/T)=|W|$.
\end{proof}

\begin{theorem}\label{T:number}
We have $e(E_\UU/E_\VV)([G/T])=C_\VV$.
\end{theorem}

\begin{proof}
By Theorem \ref{T:Borel} and Proposition \ref{P:inner},
$$f_\VV-\frac{\langle f_\VV,f_0\rangle}{\langle
f_0,f_0\rangle}f_0\in f_0^\bot=\II_m\subset\ker(\varphi).$$ So by
Theorem \ref{T:class} and Lemma \ref{L:1}, we have
\begin{align*}
&e(E_\WW)([G/T])=\varphi(f_\VV)([G/T])=\frac{\langle
f_\VV,f_0\rangle}{\langle
f_0,f_0\rangle}\varphi(f_0)([G/T])\\
=&\frac{\langle f_\VV,f_0\rangle}{\langle
f_0,f_0\rangle}e(E_\Lm)([G/T])=\frac{\langle
f_\VV,f_0\rangle}{\langle f_0,f_0\rangle}|W|=C_\VV.
\end{align*}
\end{proof}

\section{Universal subspaces}

Given the representation of $G$ in $\UU$, we say that a linear
subspace $\VV$ of $\UU$ is \emph{universal} if every $G$-orbit in
$\UU$ meets $\VV$. Let $G_\VV$ be the stabilizer of $\VV$ in $G$.

\begin{lemma}\label{L:optimal}
If $\VV$ is universal, then $\dim\VV\geq\dim\UU-\dim G/G_\VV.$
\end{lemma}

\begin{proof}
Let $G\times_{G_\VV}\VV$ be the vector bundle over $G/G_\VV$ whose
total space is the space of $G_\VV$-orbits in $G\times\VV$ with the
$G_\VV$-action $h\cdot(g,v)=(gh^{-1},\rho(h)(v))$. Then the map
$F:G\times\VV\rightarrow\UU$, $F(g,v)=\rho(g)(v)$ reduces to a map
$\widetilde{F}:G\times_{G_\VV}\VV\rightarrow\UU$. Since $\VV$ is
universal, $F$ and $\widetilde{F}$ are surjective. So $\dim
G/G_\VV+\dim\VV=\dim(G\times_{G_\VV}\VV)\geq\dim\UU$.
\end{proof}

Assume that $\VV$ is $T$-invariant, that is, $T\subset G_\VV$. By
the above lemma, if $\VV$ is universal, then $\dim \VV\geq\dim
\UU-2m$, where $m=\frac{1}{2}\dim G/T$. We say that $\VV$ has
\emph{optimal dimension} if $\dim \VV=\dim \UU-2m$. In this case the
characteristic polynomial $f_\VV$ of $\VV$ has degree $d=m$.

\begin{theorem}\label{T:main}
Let $\rho$ be a representation of $G$ in a finite dimensional real
vector space $\UU$, and let $\VV$ be a $T$-invariant subspace of
$\UU$ such that $\UU^T\subset\VV$. Then $f_\VV\notin\II$ implies
that $\VV$ is universal. In particular, if $\VV$ has optimal
dimension and $C_\VV\neq0$, then $\VV$ is universal.
\end{theorem}

\begin{proof}
Let $E_\UU,E_\VV$ be the vector bundles defined in Section 3. Every
vector $u\in\UU$, viewed as a constant section of $E_\UU$, induces a
smooth section $s_u$ of $E_\UU/E_\VV$. It is obvious that $gT$ is a
zero of $s_u$ if and only if $u\in\rho(g)(\VV)$, that is,
$\rho(g^{-1})(u)\in\VV$. So $\VV$ is universal if and only if for
any $u\in\UU$, $s_u$ has a zero. This will be ensured if the Euler
class $e(E_\UU/E_\VV)$ is nonzero.

By Theorem \ref{T:class}, $e(E_\UU/E_\VV)=\varphi(f_\VV)$, where
$\varphi:\SSS\rightarrow H^*(G/T,\RR)$ is the homomorphism
\eqref{E:homo}. By Theorem \ref{T:Borel}, the kernel of $\varphi$ is
$\II$. So $e(E_\UU/E_\VV)\neq0$ if and only if $f_\VV\notin\II$.
This proves the main assertion. The particular case follows from
Proposition \ref{P:inner}.
\end{proof}

In Theorem \ref{T:main}, when $\VV$ has optimal dimension, the
sufficient condition $C_\VV\neq0$ for the universality of $\VV$ is
fairly easy to check. If $\dim \VV>\dim \UU-2m$, then $\deg
f_\VV<m$. In this case, to determine whether $f_\VV\notin\II$, one
can use Theorem \ref{T:invariant} (5). Since $\II$ is generated by
the basic generators $F_1,\ldots,F_r$, one can also use the theory
of Gr\"{o}bner bases (see \cite{CLO}). But for many representations,
this case can be reduced to the case of optimal dimension as the
next proposition shows. (In a special case, this result was pointed
out to us by J.-K. Yu.)

\begin{proposition}\label{P:shrink}
Suppose that the identity component of $\ker(\rho)$ is a torus, and
suppose that there exists a $G$-orbit $\OO$ in $\UU$ such that $\dim
\OO=\dim G/\ker(\rho)$. If $\VV$ is a $T$-invariant subspace of
$\UU$ such that $\UU^T\subset\VV$ and $f_\VV\notin\II$, then there
exists a $T$-invariant subspace $\VV'$ of $\UU$ with optimal
dimension such that $\UU^T\subset\VV'\subset\VV$ and
$f_{\VV'}\notin\II$.
\end{proposition}

\begin{proof}
By replacing $G$ with $G/\ker(\rho)$ if necessary, we may assume
that $\rho$ is faithful. If $\VV$ has optimal dimension, there is
nothing to prove. So we assume that $d=\deg f_\VV<m$. Let
$Z=\{\beta\in\Lt^*|f_\VV \beta\in\II\}$. Then $Z$ is a subspace of
$\Lt^*$. We claim that $Z\neq\Lt^*$. Indeed, if $Z=\Lt^*$, then
$f_\VV\Lt^*\subset\II$, and then $f_\VV\SSS_{m-d}\subset\II$. Hence
$\varphi(f_\VV)H^{2(m-d)}(G/T,\RR)=\varphi(f_\VV)\varphi(\SSS_{m-d})=0$.
But by the Poincar\'{e} duality, the pairing $H^{2d}(G/T,\RR)\times
H^{2(m-d)}(G/T,\RR)\rightarrow H^{2m}(G/T,\RR)\cong\RR$ defined by
$(\xi_1,\xi_2)\mapsto\xi_1\xi_2$ is nondegenerate. As
$\varphi(f_\VV)\neq0$, we have a contradiction.

We decompose $\VV$ as $\VV=\UU^T\oplus\bigoplus_{i=1}^{d'}\VV_i$,
where each $\VV_i$ is $2$-dimensional and $T$-invariant, and the
$T$-action on $\VV_i$ is equivalent to the real $T$-module
associated to some $\chi'_i\in\XX(T)$. Since $f_\VV\notin\II$, $\VV$
is universal. So $\OO\cap\VV\neq\emptyset$ and we can choose
$v\in\OO\cap\VV$. Then for any $H\in\bigcap_{i=1}^{d'}\ker(\chi'_i)$
and $t\in\RR$, we have $\rho(e^{tH})(v)=v$. But $\dim \OO=\dim G$
forces $H=0$. So we have $\bigcap_{i=1}^{d'}\ker(\chi'_i)=0$, that
is, $\spa\{\chi'_1,\ldots,\chi'_{d'}\}=\Lt^*$. Since $Z\neq\Lt^*$,
there exists $i_0$ such that $\chi'_{i_0}\notin Z$, that is,
$f_\VV\chi'_{i_0}\notin\II$. Then, for
$\VV_1=\UU^T\oplus\bigoplus_{i\neq i_0}\VV_i$, we have
$f_{\VV_1}=f_\VV\chi'_{i_0}\notin\II$. Now the proposition follows
by induction.
\end{proof}

\begin{remark}
By the principal orbit theorem (see e.g. \cite{Kaw}), the principal
orbits have the maximal dimension and their union is an open and
dense subset of $\UU$. Hence there exists a $G$-orbit $\OO$ with
$\dim \OO=\dim G/\ker(\rho)$ if and only if the dimension of the
principal orbits is $\dim G/\ker(\rho)$. Many representations
satisfy this condition. For instance, the condition is satisfied by
the complexified adjoint representation of $G$. Indeed, there are
only finitely many real irreducible representations of $G$ for which
the condition fails. Moreover, in the case when $G$ is simple, a
complete list of real irreducible representations for which the
condition fails is given in \cite{HH}. A similar list for complex
representations (not necessarily irreducible) of complex simple Lie
groups is given in \cite{El}.
\end{remark}

It is interesting to ask to what extent the converse of Theorem
\ref{T:main} holds. We now prove a special case. Further discussion
will be given in Section 6.

Suppose that $\VV$ has optimal dimension, and let $\WW$ be a
$T$-invariant subspace of $\UU$ complementary to $\VV$. Let $\Lg$
and $\Lm$ be as before. For every $v\in\VV$, we define a map
$\psi_v:\Lm\rightarrow\WW$ by
\begin{equation}\label{E:psi}
\psi_v(X)=-P_\WW d\rho(X)(v),
\end{equation}
where $P_\WW$ is the projection of $\UU$ on $\WW$ along $\VV$, and
$d\rho:\Lg\rightarrow\mathrm{End}(\UU)$ is the differential of
$\rho$. Note that since $\VV$ has optimal dimension, we have
$\dim\Lm=\dim\WW$. Let $\det(\psi_v)$ be the determinant of the
matrix of $\psi_v$ with respect to the basis \eqref{E:basis-m} of
$\Lm$ and the basis
\begin{equation}\label{E:basis-Wm}
\{w_1,J_\WW w_1,\ldots,w_m,J_\WW w_m\}
\end{equation}
of $\WW$ (see \eqref{E:basis-W}). The map $v\mapsto\det(\psi_v)$ is
a real polynomial function on $\VV$.

\begin{theorem}\label{T:inverse}
Let the notation be as above. Suppose that $\det(\psi_v)$ does not
take both positive and negative
values on $\VV$. Then the following statements are equivalent:\\
(1) $C_\VV\neq0$;\\
(2) $\VV$ is universal;\\
(3) $\det(\psi_v)$ is not identically zero.
\end{theorem}

Before proving the theorem, we first make some preparations. Recall
from the proof of Theorem \ref{T:main} that we associated to each
$u\in\UU$ a section $s_u$ of the oriented bundle $E_\UU/E_\VV$ over
$G/T$, and that a point $gT\in G/T$ lies in the zero locus $Z(s_u)$
of $s_u$ if and only if $\rho(g^{-1})(u)\in\VV$. We denote the
$G$-orbit of $u$ by $\OO_u$.

\begin{lemma}
Let $u\in\UU$, $x=gT\in Z(s_u)$ and
$v=\rho(g^{-1})(u)$. Then\\
(1) $\ind_x(s_u)=\sgn(\psi_v)$;\\
(2) $\OO_u\pitchfork_v\VV\Leftrightarrow
s_u\pitchfork_xs_0\Leftrightarrow\det(\psi_v)\neq0$. In particular,
$\OO_u\pitchfork\VV$ if and only if $s_u\pitchfork s_0$.
\end{lemma}

\begin{proof}
(1). Since $E_\UU/E_\VV$ is equivalent to $E_\WW$, $s_u$ can be
viewed as a map $s_u:G/T\rightarrow\UU$ and can be expressed as
$$s_u(gT)=P_{\rho(g)(\WW)}(u)=\rho(g)P_\WW\rho(g^{-1})(u),$$ where
$P_{\rho(g)(\WW)}$ is the projection of $\UU$ onto $\rho(g)(\WW)$
along $\rho(g)(\VV)$. Since $x\in Z(s_u)$, we can view the tangent
map $(ds_u)_x$ as a map from $T_x(G/T)$ into $\rho(g)(\WW)$. Let
$\tau_g$ be the map defined in \eqref{E:tau}. Then we have
$$s_u\circ\tau_g(X)=\rho(g)\rho(e^X)P_\WW\rho(e^{-X})\rho(g^{-1})(u)
=\rho(g)e^{d\rho(X)}P_\WW e^{-d\rho(X)}(v).$$ Since $v\in\VV$,
\begin{align*}
&(ds_u)_x\circ (d\tau_g)_0(X)=\rho(g)(d\rho(X)P_\WW(v)-P_\WW
d\rho(X)(v))\\
=& -\rho(g)P_\WW d\rho(X)(v)=\rho(g)\psi_v(X),
\end{align*}
that is, $(ds_u)_x\circ (d\tau_g)_0=\rho(g)\circ\psi_v$. Since both
$(d\tau_g)_0$ and $\rho(g):\WW\rightarrow\rho(g)(\WW)$ are
invertible and preserve the orientation, we have
$\ind_x(s_u)=\sgn((ds_u)_x)=\sgn(\psi_v)$.

(2). The assertion
$s_u\pitchfork_xs_0\Leftrightarrow\det(\psi_v)\neq0$ follows
directly from (1), and the assertion
$\OO_u\pitchfork_v\VV\Leftrightarrow\det(\psi_v)\neq0$ is obvious
from the definition of $\psi_v$.
\end{proof}

We say that a flag $gT\in G/T$ \emph{sends $u$ into $\VV$} if $gT\in
Z(s_u)$. By the above lemma, if $\OO_u\pitchfork\VV$, then the
number
$$N(u,\VV)=\#Z(s_u)$$ of flags sending $u$ into $\VV$ is finite. By
the transversality theorem (see \cite{GP}), the vectors $u\in\UU$
for which $\OO_u\pitchfork\VV$ form an open and dense subset of
$\UU$. So $N(u,\VV)$ is finite for generic $u\in\UU$. In general,
$N(u,\VV)$ is not constant even for generic $u$. But this is the
case if the condition of Theorem \ref{T:inverse} holds.

\begin{lemma}\label{L:number}
Under the condition of Theorem \ref{T:inverse}, for $u\in\UU$ with
$\OO_u\pitchfork\VV$, we have $N(u,\VV)=|C_\VV|$.
\end{lemma}

\begin{proof}
Without loss of generality, we may assume that $\det(\psi_v)\geq0$
for every $v\in\VV$. Since $\OO_u\pitchfork\VV$, for every $x=gT\in
Z(s_u)$ and $v=\rho(g^{-1})(u)\in\VV$, we have $\det(\psi_v)\neq0$,
which implies that $\det(\psi_v)>0$. So for such $x$ and $v$,
$\ind_x(s_u)=\sgn(\psi_v)=1$. By Theorem \ref{T:number}, we have
\begin{align*}
&N(u,\VV)=\#Z(s_u)=\sum_{x\in
Z(s_u)}\ind_x(s_u)\\
=&I(s_u,s_0)=e(E_\UU/E_\VV)([G/T])=C_\VV.
\end{align*}
\end{proof}

\begin{proof}[Proof of Theorem \ref{T:inverse}]
``(1) $\Rightarrow$ (2)". This follows directly from Theorem
\ref{T:main}.

``(2) $\Rightarrow$ (3)". Since the map $G\times\UU\rightarrow\UU$,
$(g,u)\mapsto\rho(g)(u)$ is transversal to $\VV$, by the
transversality theorem, there exists $u\in\UU$ such that
$\OO_u\pitchfork\VV$. Since $\VV$ is universal,
$\OO_u\cap\VV\neq\emptyset$. So we may choose $v\in \OO_u\cap\VV$
and we have $\OO_u\pitchfork_v\VV$. So $\det(\psi_v)\neq0$.

``(3) $\Rightarrow$ (1)". Choose $v_0\in\VV$ such that
$\det(\psi_{v_0})\neq0$. Then $\OO_{v_0}\pitchfork_{v_0}\VV$.
Consider the map $F:G\times\VV\rightarrow\UU$, $F(g,v)=\rho(g)(v)$.
Then $(dF)_{(e,v_0)}$ is surjective. Hence the image $\im(F)$ of $F$
contains an open neighborhood of $v_0$ in $\UU$. By the
transversality theorem, we can choose $u\in\im(F)$ such that
$\OO_u\pitchfork\VV$. Since $u\in\im(F)$,
$\OO_u\cap\VV\neq\emptyset$. So $Z(s_u)$ is not empty.  By Lemma
\ref{L:number}, $|C_\VV|=N(u,\VV)>0$.
\end{proof}

Now we give an example for which the assumption of Theorem
\ref{T:inverse} holds. Let $\Lm_\CC$ be the complexification of
$\Lm$, and let $\Ln^\pm=\bigoplus_{\alpha\in\Phi^\pm}\Lg_\alpha$,
where $\Phi^-=-\Phi^+$. Then $\Lm_\CC=\Ln^+\oplus\Ln^-$. Recall that
$\WW$ can be equipped with the complex structure $J_\WW$ as in
\eqref{E:JW}. Then the map $\psi_v:\Lm\rightarrow\WW$ extends
uniquely to a complex linear map
$(\psi_v)_\CC:\Lm_\CC\rightarrow(\WW,J_\WW)$ for all $v\in\VV$.

\begin{remark}\label{R:2}
As in Remark \ref{R:1}, if $\UU$ admits a complex structure such
that the $G$-action is complex linear, and if $\VV$, $\WW$ are
complex subspaces of $\UU$, then $J_\WW$ can be chosen to be the
same as the original complex structure on $\WW$. In this case, the
homomorphism $d\rho:\Lg\rightarrow\mathrm{End}_\CC(\UU)$ extends
uniquely to a complex homomorphism
$(d\rho)_\CC:\Lg_\CC\rightarrow\mathrm{End}_\CC(\UU)$, and we have
$$
(\psi_v)_\CC(X)=-P_\WW (d\rho)_\CC(X)(v)
$$
for $v\in\VV$ and $X\in\Lm_\CC$.
\end{remark}

\begin{corollary}\label{C:inverse}
Let the notation be as above. Suppose that, with respect to the
complex structure $J_\WW$, there is a complex decomposition
$\WW=\WW^+\bigoplus\WW^-$ such that
$(\psi_v)_\CC(\Ln^\pm)\subset\WW^\pm$ for all $v\in\VV$. Then $\VV$
is universal if and only if $C_\VV\neq0$. Furthermore, if $u\in\UU$
and $\OO_u\pitchfork\VV$, then $N(u,\VV)=|C_\VV|$.
\end{corollary}

\begin{proof}
By Theorem \ref{T:inverse} and Lemma \ref{L:number}, it suffices to
show that $\det(\psi_v)$ does not take both positive and negative
values on $\VV$.

Let $J'_\WW$ be the complex structure on $\WW$ defined by
$$\begin{cases}J'_\WW(w)=J_\WW(w), & w\in\WW^+;\\J'_\WW(w)=-J_\WW(w), &
w\in\WW^-.\end{cases}$$ Then for $\alpha\in\Phi^+$,
\begin{align*}
\psi_vJ_\Lm(X_\alpha-X_{-\alpha})=&\psi_v(\sqrt{-1}(X_\alpha+X_{-\alpha}))\\
=&(\psi_v)_\CC(\sqrt{-1}X_\alpha)+(\psi_v)_\CC(\sqrt{-1}X_{-\alpha})\\
=&J_\WW(\psi_v)_\CC(X_\alpha)+J_\WW(\psi_v)_\CC(X_{-\alpha})\\
=&J'_\WW(\psi_v)_\CC(X_\alpha)-J'_\WW(\psi_v)_\CC(X_{-\alpha})\\
=&J'_\WW\psi_v(X_\alpha-X_{-\alpha}),
\end{align*}
\begin{align*}
\psi_vJ_\Lm(\sqrt{-1}(X_\alpha+X_{-\alpha}))=&-\psi_v(X_\alpha-X_{-\alpha})\\
=&J'_\WW J'_\WW\psi_v(X_\alpha-X_{-\alpha})\\
=&J'_\WW\psi_vJ_\Lm(X_\alpha-X_{-\alpha})\\
=&J'_\WW\psi_v(\sqrt{-1}(X_\alpha+X_{-\alpha})).
\end{align*}
So the map $\psi_v:\Lm\rightarrow\WW$ is complex linear with respect
to the complex structures $J_\Lm$ and $J'_\WW$. Hence with respect
to the basis \eqref{E:basis-m} of $\Lm$ and the basis $\{w'_1,J'_\WW
w'_1,\ldots,w'_m,J'_\WW w'_m\}$ of $\WW$, the matrix of $\psi_v$ has
nonnegative determinant, where $\{w'_1,\ldots,w'_m\}$ is any complex
basis of $(\WW,J'_\WW)$. So with respect to the basis
\eqref{E:basis-m} of $\Lm$ and the basis \eqref{E:basis-Wm} of
$\WW$, either $\det(\psi_v)\geq0$ for all $v\in\VV$, or
$\det(\psi_v)\leq0$ for all $v\in\VV$.
\end{proof}

\section{Generalizations of Schur's triangularization theorem}\label{Schur}

The classical Schur triangularization theorem asserts that for the
complexified adjoint representation of $U(n)$ in $\MM(n,\CC)$, the
subspace of $\MM(n,\CC)$ consisting of all upper (or lower)
triangular matrices is universal. We generalize this theorem to a
whole class of subspaces of $\MM(n,\CC)$ by using the results in the
previous sections. We shall also obtain similar generalizations for
the conjugation representations of $Sp(n)$ and $SO(n)$ in
$\MM(n,\HH)$ and $\MM(n,\RR)$, respectively. To unify the results,
we introduce the following definitions.

\begin{definition}
Let $n$ be a positive integer. \\
(1) An \emph{$n\times n$ zero pattern} is a set
$$I\subset\{(i,j)|i,j\in\{1,\ldots,n\}, i\neq j\}$$ of cardinality
$n(n-1)/2$. \\
(2) The \emph{characteristic polynomial} $\lambda_I$ of an $n\times
n$ zero pattern $I$ is
$$\lambda_I(x)=\prod_{(i,j)\in I}(x_i-x_j)\in\RR[x_1,\ldots,x_n].$$\\
(3) The \emph{characteristic number} of $I$ is
$C_I=\langle\lambda_I,\lambda_0\rangle$, where $\lambda_0$ is the
polynomial \eqref{E:lambda0}, and $\langle\cdot,\cdot\rangle$ is the
inner product on $\RR[x_1,\ldots,x_n]$ for which the set of
monomials $\{x_1^{m_1}\cdots x_n^{m_n}\}$ is an orthonormal basis.
\end{definition}

It is obvious that $\lambda_0$ is the characteristic polynomial of
the \emph{upper triangular} zero pattern
$$I_0=\{(i,j)|1\leq i<j\leq n\},$$ and that
$C_{I_0}=\langle\lambda_0,\lambda_0\rangle=n!$.

We first consider the complex case. For $A\in\MM(n,\CC)$, let
$A_{ij}$ denote the $(i,j)$-th entry of $A$.

\begin{theorem}\label{T:A}
Consider the complexified adjoint representation of $U(n)$ in
$\MM(n,\CC)$. Let $I$ be an $n\times n$ zero pattern. If $C_I\neq0$,
then the subspace
$$\MM_I(n,\CC)=\{A\in\MM(n,\CC)|A_{ij}=0 \
\text{for} \ (i,j)\in I\}$$ is universal.
\end{theorem}

\begin{proof}
We keep the notation from Example \ref{Ex:A}. For $(i,j)\in I$, let
$$\WW_{ij}=\{A\in\MM(n,\CC)|A_{kl}=0 \ \text{if} \ (k,l)\neq(i,j)\}.$$
Then $\WW_{ij}$ is $T$-invariant and $T$ acts on $\WW_{ij}$ via the
character $x_i-x_j$ (identified with $t\mapsto t_it_j^{-1}$). The
subspace $\WW=\bigoplus_{(i,j)\in I}\WW_{ij}$ of $\MM(n,\CC)$ is
complementary to $\VV=\MM_I(n,\CC)$. Note that $\VV$ has optimal
dimension. By Remark \ref{R:1}, we have
$$f_{\VV}(x)=\prod_{(i,j)\in I}(x_i-x_j)=\lambda_I(x),$$ and
$$C_{\VV}=\frac{\langle f_{\VV},f_0\rangle}{\langle
f_0,f_0\rangle}|S_n|=\frac{\langle
\lambda_I,\lambda_0\rangle}{\langle
\lambda_0,\lambda_0\rangle}n!=C_I.$$ So by Theorem \ref{T:main},
$\VV$ is universal.
\end{proof}

Note that $\MM_{I_0}(n,\CC)$ consists of all lower triangular
matrices. So if $I=I_0$, the above theorem reduces to the classical
Schur triangularization theorem.

We say that a zero pattern $I$ is \emph{bitriangular} if
$(i_0,j_0)\in I$ and $(j_0-i_0)(i_0-j_0+j-i)>0$ imply that $(i,j)\in
I$. If $i_0<j_0$ (resp. $i_0>j_0$), this condition says that
$(i,j)\in I$ whenever $j-i>j_0-i_0$ (resp. $i-j>i_0-j_0$). Using
Corollary \ref{C:inverse}, we can prove the following theorem.

\begin{theorem}\label{T:A-inverse}
Let $I$ be a bitriangular zero pattern. Then $\MM_I(n,\CC)$ is
universal if and only if $C_I\neq0$. Furthermore, if the
$U(n)$-orbit of $A\in\MM(n,\CC)$ is transversal to $\MM_I(n,\CC)$,
then $N(A,\MM_I(n,\CC))=|C_I|$.
\end{theorem}

\begin{proof}
In the notation of Corollary \ref{C:inverse}, $\Ln^+$ (resp.
$\Ln^-$) is the set of all strictly upper (resp. lower) triangular
matrices in $\MM(n,\CC)$. Let $\VV=\MM_I(n,\CC)$, and let $\WW$ be
the subspace defined in the proof of Theorem \ref{T:A} and
$\WW^\pm=\WW\cap\Ln^\pm$. By Remark \ref{R:2}, for $v\in\VV$, we
have $(\psi_v)_\CC(X)=P_\WW([v,X])$, $X\in\Lm_\CC$. Since $I$ is
bitriangular, it is easy to see that
$P_\WW([\VV,\Ln^\pm])\subset\WW^\pm$. So
$(\psi_v)_\CC(\Ln^\pm)\subset\WW^\pm$ for all $v\in\VV$, and the
theorem follows from Corollary \ref{C:inverse}.
\end{proof}

\begin{remark}
For the complexified adjoint representation of $U(n)$ in
$\MM(n,\CC)$, the number $N(A,\MM_I(n,\CC))$ of flags sending $A$
into $\MM_I(n,\CC)$ can be interpreted as follows. By viewing $A$ as
an operator on $\CC^n$ and a flag as an ordered $n$-tuple
$(l_1,\ldots,l_n)$ of mutually orthogonal complex lines in $\CC^n$,
$N(A,\MM_I(n,\CC))$ is the number of flags $(l_1,\ldots,l_n)$ such
that with respect to an ordered basis $\{e_1,\ldots,e_n\}$ of
$\CC^n$ with $e_i\in l_i$, the matrix of $A$ is in $\MM_I(n,\CC)$.
\end{remark}

Now we consider the quaternionic case. For a matrix
$A\in\MM(n,\HH)$, let $A_{ij}$ denote the $(i,j)$-th entry of $A$.

\begin{theorem}\label{T:C}
Consider the conjugation representation of $Sp(n)$ in $\MM(n,\HH)$.
Let $I$ be an $n\times n$ zero pattern. If $C_I\neq0$, then the
subspace
$$\MM_I(n,\HH)=\{A\in\MM(n,\HH)|A_{ij}=0 \ \text{for} \ (i,j)\in I\}$$
is universal.
\end{theorem}

\begin{proof}
We keep the notation from Example \ref{Ex:C}. Let $\VV$ be the
subspace consisting of all matrices $A\in\MM_I(n,\HH)$ such that
$A_{ii}\in\CC$ for $1\leq i\leq n$. For $(i,j)\in I$, let
$$\WW_{ij}=\{A\in\MM(n,\HH)|A_{kl}=0 \ \text{if} \
(k,l)\neq(i,j)\},$$ and let
$$\WW_i=\{A\in\MM(n,\HH)|A_{ii}\in \mathbf{j}\CC
\ \text{and} \ A_{kl}=0 \ \text{for} \ (k,l)\neq(i,i)\}$$ for $1\leq
i\leq n$. Then the subspace $$\WW=\left(\bigoplus_{(i,j)\in
I}\WW_{ij}\right) \oplus\left(\bigoplus_{i=1}^n\WW_i\right)$$ is
complementary to $\VV$ in $\MM(n,\HH)$. It is easy to see that each
$T$-invariant subspace $\WW_{ij}$ can be decomposed as the direct
sum of two $T$-irreducible subspaces and, with respect to suitable
orientations, they are equivalent to the characters $x_i+x_j$ and
$x_i-x_j$ of $T$. Moreover, each $\WW_i$ is $T$-irreducible and,
with respect to suitable orientation, it is equivalent to the
character $2x_i$ of $T$. So we have
$$f_{\VV}(x)=2^n\prod_{(i,j)\in I}(x_i^2-x_j^2)\prod_{i=1}^nx_i
=2^n\lambda_I(x_1^2,\ldots,x_n^2)\prod_{i=1}^nx_i.$$ Note that $\VV$
has optimal dimension, and the endomorphism of $\RR[x_1,\ldots,x_n]$
sending $f(x_1,\ldots,x_n)$ to
$f(x_1^2,\ldots,x_n^2)\prod_{i=1}^nx_i$ is isometric. So
$$C_{\VV}=\frac{\langle f_{\VV},f_0\rangle}{\langle
f_0,f_0\rangle}|(\ZZ/2\ZZ)^n\rtimes S_n|=\frac{\langle
\lambda_I,\lambda_0\rangle}{\langle
\lambda_0,\lambda_0\rangle}2^nn!=2^nC_I.$$ By Theorem \ref{T:main},
if $C_I\neq0$ then $\VV$ is universal, and \textit{a fortiori}
$\MM_I(n,\HH)$ is universal.
\end{proof}

For the real case, we first consider the case of $\MM(2n,\RR)$. We
partition a matrix $A\in\MM(2n,\RR)$ into $2\times 2$ blocks and
denote its $(i,j)$-th block entry by $a_{ij}(A)$.

\begin{theorem}\label{T:D}
Consider the conjugation representation of $SO(2n)$ in
$\MM(2n,\RR)$. Let $I$ be an $n\times n$ zero pattern. If
$C_I\neq0$, then the subspace
$$\MM_I(2n,\RR)=\{A\in\MM(2n,\RR)|a_{ij}(A)=0 \ \text{for} \ (i,j)\in I\}$$
is universal.
\end{theorem}

\begin{proof}
We keep the notation from Example \ref{Ex:D}. For $(i,j)\in I$, let
$$\WW_{ij}=\{A\in\MM(2n,\RR)|a_{kl}(A)=0 \ \text{if} \
(k,l)\neq(i,j)\}.$$ Then $\WW=\bigoplus_{(i,j)\in I}\WW_{ij}$ is a
subspace of $\MM(2n,\RR)$ complementary to $\VV=\MM_I(2n,\RR)$. As
in the proof of Theorem \ref{T:C}, each $\WW_{ij}$ decomposes as the
direct sum of two $T$-irreducible subspaces, and with suitable
orientations, they are orientation preserving equivalent to the
characters $x_i+x_j$ and $x_i-x_j$ of $T$, respectively. So we have
$$f_{\VV}(x)=\prod_{(i,j)\in I}(x_i^2-x_j^2)=\lambda_I(x_1^2,\ldots,x_n^2).$$
Note that $\VV$ has optimal dimension, and the endomorphism of
$\RR[x_1,\ldots,x_n]$ sending $f(x_1,\ldots,x_n)$ to
$f(x_1^2,\ldots,x_n^2)$ is isometric. So
$$C_{\VV}=\frac{\langle f_{\VV},f_0\rangle}{\langle
f_0,f_0\rangle}|(\ZZ/2\ZZ)^{n-1}\rtimes S_n|=\frac{\langle
\lambda_I,\lambda_0\rangle}{\langle
\lambda_0,\lambda_0\rangle}2^{n-1}n!=2^{n-1}C_I.$$ By Theorem
\ref{T:main}, if $C_I\neq0$ then $\VV_I$ is universal.
\end{proof}

For the case of $\MM(2n+1,\RR)$, we partition a matrix
$A\in\MM(2n+1,\RR)$ into blocks $a_{ij}(A)$ of size
$(2-\delta_{i,n+1})\times(2-\delta_{j,n+1})$, where $\delta_{i,j}$
is the Kronecker symbol.

\begin{theorem}\label{T:B}
Consider the conjugation representation of $SO(2n+1)$ in
$\MM(2n+1,\RR)$. Let $J$ be an $(n+1)\times (n+1)$ zero pattern such
that exactly one of $(i,n+1)$ and $(n+1,i)$ is in $J$ for any $i\le
n$, and let $I=\{(i,j)\in J | i,j\le n \}$. If $C_I\neq0$, then the
subspace
$$ \MM_{J}(2n+1,\RR)=\{A\in\MM(2n+1,\RR)|a_{ij}(A)=0 \ \text{for} \
(i,j)\in J\} $$ is universal.
\end{theorem}

\begin{proof}
The subspace
$$ \WW=\{A\in\MM(2n+1,\RR)|a_{ij}(A)=0 \ \text{for} \ (i,j)\notin J\} $$
of $\MM(2n+1,\RR)$ is complementary to $\VV=\MM_{J}(2n+1,\RR)$. By
choosing a decomposition of $\WW$ and suitable orientations on the
direct summands, we compute
$$f_{\VV}(x)=\prod_{(i,j)\in I}(x_i^2-x_j^2)\prod_{i=1}^nx_i
=\lambda_I(x_1^2,\ldots,x_n^2)\prod_{i=1}^nx_i,$$ and
$$C_{\VV}=\frac{\langle f_{\VV},f_0\rangle}{\langle
f_0,f_0\rangle}|(\ZZ/2\ZZ)^n\rtimes S_n|=\frac{\langle
\lambda_I,\lambda_0\rangle}{\langle
\lambda_0,\lambda_0\rangle}2^nn!=2^nC_I.$$ So if $C_I\neq0$, then by
Theorem \ref{T:main}, $\VV$ is universal.
\end{proof}

\begin{remark}
It is easy to see that the condition of Proposition \ref{P:shrink}
is satisfied for all representations considered in this section.
\end{remark}

We say that an $n\times n$ zero pattern $I$ is \emph{simple} if $I$
contains exactly one of $(i,j)$ and $(j,i)$ for any
$i,j\in\{1,\ldots,n\}$ with $i\neq j$.

\begin{theorem}\label{T:A2}
Let $I$ (resp. $J$) be a simple $n\times n$ (resp.
$(n+1)\times(n+1)$) zero pattern. Then in the context of Theorem
\ref{T:A} (resp. \ref{T:C}, \ref{T:D}, \ref{T:B}) the subspace
$\MM_I(n,\CC)$ (resp. $\MM_I(n,\HH)$, $\MM_I(2n,\RR)$,
$\MM_J(2n+1,\RR)$) is universal.
\end{theorem}

\begin{proof}
We have
$$
\lambda_I(x)=\prod_{(i,j)\in I}(x_i-x_j)=\pm\prod_{(i,j)\in
I_0}(x_i-x_j)=\pm\lambda_0(x).
$$
So $C_I=\langle\lambda_I,\lambda_0\rangle=
\pm\langle\lambda_0,\lambda_0\rangle=\pm n!$, and the assertion
follows from the above theorems.
\end{proof}

The problem of universality of the subspace in the following example
was raised in \cite{DJ} but seems difficult to prove by other
methods.

\begin{example}\label{Ex:cyclic}
Let $n=3$. Consider the cyclic zero pattern
$$I=\{(1,3),(2,1),(3,2)\}.$$ Since $I$ is a simple zero pattern, $\MM_I(3,\CC)$ is
universal with respect to the conjugation representation of $U(3)$
in $\MM(3,\CC)$. \qed
\end{example}

\begin{example}[\cite{VP,DD}]
Let $n=4$. Consider the bitriangular zero pattern
$$I=\{(1,3),(1,4),(2,4),(3,1),(4,1),(4,2)\}.$$ Its characteristic
polynomial is
$$\lambda_I(x)=-(x_1-x_3)^2(x_1-x_4)^2(x_2-x_4)^2,$$ and then
$C_I=\langle\lambda_I,\lambda_0\rangle=-12$. So $\MM_I(4,\CC)$ is
universal with respect to the conjugation representation of $U(4)$
in $\MM(4,\CC)$. Since $I$ is bitriangular, for $A\in\MM(n,\CC)$
with $\OO_A\pitchfork\MM_I(4,\CC)$ we have $N(A,\MM_I(4,\CC))=12$.
\qed
\end{example}

Using the same method, it is easy to generalize Theorems
\ref{T:A-inverse} and \ref{T:A2} to arbitrary complex
representations of any connected compact Lie group. We state the
result for the complexified adjoint representation below and leave
the details to the reader.

\begin{theorem}
Consider the complexified adjoint representation of a connected
compact Lie group $G$ in $\Lg_\CC$. Let $\Psi$ be a subset of the
root system $\Phi$ such that $|\Psi|=|\Phi^+|$, and let
$\VV=\Lt_\CC\oplus\bigoplus_{\alpha\in\Phi\setminus\Psi}\Lg_\alpha$. \\
(1) If $\Phi=\Psi\cup(-\Psi)$, then $\VV$ is universal;\\
(2) If $\alpha\in\Psi\cap\Phi^\pm$ and $\beta\in\Phi^\pm$ (with the
same choice of signs) imply that
$\alpha+\beta\notin\Phi\setminus\Psi$, then $\VV$ is universal if
and only if $C_\VV\neq0$. Furthermore, if $X\in\Lg_\CC$ and
$\OO_X\pitchfork\VV$, then $N(X,\VV)=|C_\VV|$.
\end{theorem}

\section{Concluding remark}

In the previous sections, we assumed that $T\subset G_\VV$, and then
constructed the vector bundles $E_\WW\cong E_\UU/E_\VV$ over $G/T$,
and reduced the question of universality of $\VV$ to the existence
of zeros of certain sections of $E_\WW$. A sufficient condition for
this is that the Euler class $e(E_\WW)$ is nonzero.

This method remains valid if we replace $T$ by any closed subgroup
$H$ of $G_\VV$. The condition that $T\subset G_\VV$ can be also
dropped. More precisely, for any such $H$, we let $\WW$ be a
$G_\VV$-invariant subspace of $\UU$ complementary to $\VV$, and
construct the subbundles $E_{H,\VV}$ and $E_{H,\WW}$ of the trivial
bundle $E_{H,\UU}=G/H\times\UU$ whose fibers at $gH$ are
$\rho(g)(\VV)$ and $\rho(g)(\WW)$, respectively. Similarly to
Theorem \ref{T:main}, if certain characteristic class $o(E_{H,\WW})$
(e.g. Euler, top Stiefel-Whitney, or top Chern class) of $E_{H,\WW}$
is nonzero, which ensures that every smooth section of $E_{H,\WW}$
has a zero, then $\VV$ is universal. For example, if $H$ preserves
an orientation on $\WW$, then $E_{H,\WW}$ is orientable and we can
use the Euler class.

If $H_1\subset H_2$ are two closed subgroups of $G_\VV$, then there
is a natural quotient map $q:G/H_1\rightarrow G/H_2$, and it is easy
to see that $E_{H_1,\WW}\cong q^*(E_{H_2,\WW})$. Hence we have
$o(E_{H_1,\WW})=q^*(o(E_{H_2,\WW}))$. If $q^*$ is not injective, it
may happen that $o(E_{H_1,\WW})=0$ but $o(E_{H_2,\WW})\neq0$. So by
using a larger subgroup $H$ of $G_\VV$, we may get stronger
sufficient condition for the universality of $\VV$.

An interesting case occurs when $\dim\VV=\dim\UU-\dim G/G_\VV$ and
$T\subset G_\VV$. We further assume that $T$ is the identity
component of $G_\VV$. Then $G_\VV/T$ is a subgroup of the Weyl group
$W$. To avoid the trouble with non-orientability, we assume that
$\UU$ admits a complex structure such that the $G$-action is complex
linear, and that $\VV$ and $\WW$ are complex subspaces. Then we can
use the top Chern class (which is equal to the Euler class of the
underlying real bundle). We choose $H_1=T$, $H_2=G_\VV$. Then $\dim
G/T=\dim G/G_\VV=2m$, and $c_{\mathrm{top}}(E_{T,\WW})\in
H^{2m}(G/T,\ZZ)$, $c_{\mathrm{top}}(E_{G_\VV,\WW})\in
H^{2m}(G/G_\VV,\ZZ)$. It is easy to prove that $G/G_\VV$ is
non-orientable if and only if $G_\VV/T$ contains an odd element of
$W$. From algebraic topology, if $G/G_\VV$ is orientable, then
$H^{2m}(G/G_\VV,\ZZ)\cong\ZZ$ and
$q^*:H^{2m}(G/G_\VV,\ZZ)\rightarrow H^{2m}(G/T,\ZZ)$ is injective.
So we obtain the same information from the top Chern classes by
using $T$ and $G_\VV$. But if $G/G_\VV$ is non-orientable, then
$H^{2m}(G/G_\VV,\ZZ)\cong\ZZZ$ and
$q^*:H^{2m}(G/G_\VV,\ZZ)\rightarrow H^{2m}(G/T,\ZZ)$ is equal to
zero. In this case we always have
$c_{\mathrm{top}}(E_{T,\WW})=q^*(c_{\mathrm{top}}(E_{G_\VV,\WW}))=0$,
but it may happen that $c_{\mathrm{top}}(E_{G_\VV,\WW})\neq0$, and
if this is the case, we get more information than by using $T$ and
we can prove the universality of $\VV$. We demonstrate this in the
following simple example.

Consider the complexified adjoint representation of $G=SU(2)$ in
$\UU=\mathfrak{sl}(2,\CC)$. Let $\VV=\left\{\begin{pmatrix}0&*\\
*&0\end{pmatrix}\right\}$. Then $G_\VV=N_G(T)$. Let $\WW=\left\{\begin{pmatrix}a&0\\
0&-a\end{pmatrix}|a\in\CC\right\}$. Since $G_\VV/T=W$ contains an
odd element, $e(E_{T,\WW})=0$. Indeed, $E_{T,\WW}$ is the trivial
complex line bundle over $G/T\cong S^2$. But it is easy to see that
$E_{G_\VV,\WW}$, as a complex line bundle over $G/G_\VV\cong\RR
P^2$, is equivalent to the complexification of the tautological real
line bundle over $\RR P^2$, which has nonzero first Chern class. So
$c_{\mathrm{top}}(E_{G_\VV,\WW})\neq0$ and $\VV$ is universal.

The above consideration motivates us to ask the following question.

\begin{question}
Suppose that $\dim\VV=\dim\UU-\dim G/G_\VV$ and $T\subset G_\VV$.
Does the universality of $\VV$ imply that certain obstruction class
of $E_{G_\VV,\WW}$, for the existence of non-vanishing sections, is
nonzero?
\end{question}

We remark that if the condition $\dim\VV=\dim\UU-\dim G/G_\VV$ or
$T\subset G_\VV$ is dropped, then in general the answer is negative.
This can be seen from the following counter-examples.

For the case of $T\not\subset G_\VV$, consider the representation
$e^{i\theta}\mapsto\diag(e^{i\theta},e^{2i\theta})$ of $U(1)$ in
$\CC^2$. We view $\CC^2$ as a real vector space, and let
$\VV=\{(a,b)\in\CC^2|\mathrm{Re}(a+b)=0\}$. Then $G_\VV$ is trivial,
and $E_{G_\VV,\WW}$ is a trivial real line bundle over $U(1)$. But
it is easy to see that $\VV$ is universal.

For the case of $\dim\VV>\dim\UU-\dim G/G_\VV$, we can consider the
complexified adjoint representation of any semisimple compact group
$G$ of rank greater than $1$ in $\Lg_\CC$. Let $X$ be a regular
element in $\Lt$, let $\WW=\CC X$, and let $\VV=\WW^\bot$ with
respect to the Killing form of $\Lg_\CC$. Then $G_\VV=T$, and the
above inequality holds. It is easy to see that $E_{G_\VV,\WW}$ is a
trivial complex line bundle over $G/T$. But since $\VV$ contains the
sum of all root spaces in $\Lg_\CC$, $\VV$ is universal by
\cite{DT}, Theorem 3.4.

\end{document}